\documentclass{article}

\usepackage{amsmath,amsthm}
\usepackage{amssymb}

\usepackage{a4wide}
\usepackage{cite}

\newtheorem{thm}{Theorem}

\newtheorem{lem}[thm]{Lemma}
\newtheorem{prop}[thm]{Proposition}
\newtheorem{defn}[thm]{Definition}

\newtheorem{rem}[thm]{Remark}
\newtheorem{notat}[thm]{Notation}

\begin{document}

\title{Symmetric duality for left and right Riemann--Liouville\\
and Caputo fractional differences\thanks{This is a preprint of 
a paper whose final and definite form is published open access
in the \emph{Arab Journal of Mathematical Sciences} (ISSN: 1319-5166), 
{\tt http://dx.doi.org/10.1016/j.ajmsc.2016.07.001}.}}

\author{Thabet Abdeljawad$^1$\\
{\tt tabdeljawad@psu.edu.sa}
\and
Delfim F. M. Torres$^2$\thanks{Corresponding author.}\\
{\tt delfim@ua.pt}}

\date{$^1$Department of Mathematics and Physical Sciences,\\
Prince Sultan University, P. O. Box 66833, Riyadh 11586, Saudi Arabia\\[0.3cm]
$^2$\text{Center for Research and Development in Mathematics and Applications (CIDMA),}\\
Department of Mathematics, University of Aveiro, 3810-193 Aveiro, Portugal}

\maketitle


\begin{abstract}
A discrete version of the symmetric duality of Caputo--Torres, to relate left
and right Riemann--Liouville and Caputo fractional differences, is considered.
As a corollary, we provide an evidence to the fact that in case of right
fractional differences, one has to mix between nabla and delta operators.
As an application, we derive right fractional summation by parts formulas
and left fractional difference Euler--Lagrange equations for discrete
fractional variational problems whose Lagrangians
depend on right fractional differences.
\end{abstract}


\bigskip


\bigskip



\section{Introduction}

The study of differences of fractional order is a subject with
a long and rich history \cite{MR0346352,MR0605572,Gray,MR0614953,Miller}.
The topic has attracted the attention of a very active
community of researchers in the 21st century.
In \cite{Thabet}, the $Q$-operator connection between
delay-type and advanced-type equations is established and its discrete version
is used in \cite{ThCaputo,Th:dual:Caputo,Th:dual:Riemann}.
In \cite{Feri}, the fundamental elements of a theory of difference operators
and difference equations of fractional order are presented, while \cite{Nabla}
discusses basic properties of nabla fractional sums and differences; the validity
of a power rule and a law of exponents. Using such properties, a discrete
Laplace transform is studied and applied to initial value problems \cite{Nabla}.
In \cite{Atmodel}, the simplest discrete fractional problem of the calculus
of variations is defined and necessary optimality conditions of Euler--Lagrange
type derived. Moreover, a Gompertz fractional difference model for tumor growth
is introduced and solved. In \cite{MR2728463,Nuno}, the study of fractional
discrete-time variational problems of order $\alpha$, $0<\alpha \leq 1$,
involving discrete analogues of Riemann--Liouville fractional-order derivatives
on time scales, is introduced. A fractional formula for summation by parts
is proved, and then used to obtain Euler--Lagrange and Legendre type necessary
optimality conditions. The theoretical results are supported by several
illustrative examples \cite{MR2728463,Nuno}. More generally, it is also
possible to investigate fractional calculus on an arbitrary time scale
(that is, on an arbitrary nonempty closed set of the real numbers)
\cite{MR2800417,MyID:296,MyID:320,MyID:328,MR3272191}. The literature
on the discrete fractional calculus is now vast: see
\cite{MR3144750,MR3427746,MR3270279,MR3333836,MR3399266} and references therein.
For a comprehensive treatment, related topics of current interest and an extensive
list of references, we refer the interested readers to the book \cite{book:GP}.

In the recent article \cite{Caputo:Torres}, Caputo and Torres introduced
and developed a duality theory for left and right fractional derivatives,
that we call here \emph{symmetric duality}, defined by $f^*(t)=f(-t)$,
where $f$ is defined on $[a,b]$. They used this symmetric duality to relate
left and right fractional integrals and left and right fractional
Riemann--Liouville and Caputo derivatives. Here we show that the theory
of \cite{Caputo:Torres} can also be extended to the discrete fractional calculus.
As we prove here, the symmetric duality is very interesting because it confirms
and provides a solid foundation to the right discrete fractional calculus,
as done with the $Q$-operator in \cite{Th:dual:Caputo}. Indeed, in his articles
\cite{Th:dual:Caputo,Th:dual:Riemann}, Abdeljawad used the well-known
$Q$-operator, $Qf(t)=f(a+b-t)$, to relate left and right fractional sums
and left and right fractional differences within the delta and nabla operators.
Here we show that Abdeljawad's definitions \cite{Th:dual:Caputo,Th:dual:Riemann}
for right Riemann--Liouville and Caputo fractional differences are in some sense
a consequence of symmetric duality.

The paper is organized as follows. In Section~\ref{sec:prelim},
we recall necessary notions and results from the discrete
fractional calculus. Main results are then given in Section~\ref{sec:mr},
where several identities for delta and nabla fractional sums and differences
are proved from symmetric duality. We end with applications in
Sections~\ref{sec:appl1} and \ref{sec:appl2}:
in Section~\ref{sec:appl1} we prove summation by parts formulas
for right fractional differences (Theorems~\ref{Ap1} and \ref{Ap2}),
which are then used in Section~\ref{sec:appl2} to obtain left versions
of the fractional difference Euler--Lagrange equations for
discrete right fractional variational problems.


\section{Preliminaries}
\label{sec:prelim}

In this section, we review well-known definitions and essential results
from the literature of discrete fractional calculus, and we fix notations.
For a natural number $n$, the factorial polynomial is defined by
$t^{(n)}=\prod_{j=0}^{n-1} (t-j)=\frac{\Gamma(t+1)}{\Gamma(t+1-j)}$,
where $\Gamma$ denotes the special gamma function and the product is
zero when $t+1-j=0$ for some $j$. More generally, for arbitrary real $\alpha$,
we define $t^{(\alpha)}=\frac{\Gamma(t+1)}{\Gamma(t+1-\alpha)}$,
where one uses the convention that division at a pole yields zero.
Given that the forward and backward difference operators are defined by
$\Delta f(t)=f(t+1)-f(t)$ and $\nabla f(t)=f(t)-f(t-1)$, respectively,
we define iteratively the operators  $\Delta^m=\Delta(\Delta^{m-1})$
and $\nabla^m=\nabla(\nabla^{m-1})$, where $m$ is a natural number.
Follows some properties of the factorial function.

\begin{lem}[See \cite{Ferd}]
\label{pfp}
Let $\alpha \in \mathbb{R}$. Assume the following factorial
functions are well defined. Then,
$\Delta t^{(\alpha)}=\alpha  t^{(\alpha-1)}$;
$(t-\alpha)t^{(\alpha)}= t^{(\alpha+1)}$;
$\alpha^{(\alpha)}=\Gamma (\alpha+1)$;
if $t\leq r$, then $t^{(\alpha)}\leq r^{(\alpha)}$
for any $\alpha>r$;
if $0<\alpha<1$, then $ t^{(\alpha\nu)}\geq (t^{(\nu)})^\alpha$;
$t^{(\alpha+\beta)}= (t-\beta)^{(\alpha)} t^{(\beta)}$.
\end{lem}

The next two relations, the proofs of which are straightforward,
are also useful for our purposes:
$\nabla_s (s-t)^{(\alpha-1)}=(\alpha-1)(\rho(s)-t)^{(\alpha-2)}$,
$\nabla_t (\rho(s)-t)^{(\alpha-1)}=-(\alpha-1)(\rho(s)-t)^{(\alpha-2)}$,
where $\rho(s) = s-1$ is the backward jump operator.
With respect to the nabla fractional calculus,
we have the following definition.

\begin{defn}[See \cite{Adv,Boros,Grah,Spanier}]
\label{rising}
Let $m \in \mathbb{N}$, $\alpha \in \mathbb{R}$.
The $m$ rising (ascending) factorial of $t$ is defined by
$t^{\overline{m}}= \prod_{k=0}^{m-1}(t+k)$,
$t^{\overline{0}}=1$; the $\alpha$ rising function by
$t^{\overline{\alpha}}=\frac{\Gamma(t+\alpha)}{\Gamma(t)}$,
$t \in \mathbb{R} \setminus \{\ldots,-2,-1,0\}$,
with $0^{\mathbb{\overline{\alpha}}}=0$.
\end{defn}

\begin{rem}
For the rising factorial function, observe that
$\nabla (t^{\overline{\alpha}})=\alpha t^{\overline{\alpha-1}}$,
$(t^{\overline{\alpha}})=(t+\alpha-1)^{(\alpha)}$, and
$\Delta_t (s-\rho(t))^{\overline{\alpha}}
= -\alpha  (s-\rho(t))^{\overline{\alpha-1}}$.
\end{rem}

\begin{notat}
Along the text, we use the following notations.
\begin{description}
\item[$(i)$] For a real $\alpha>0$, we set $n=[\alpha]+1$,
where $[\alpha]$ is the greatest integer less than $\alpha$.

\item[$(ii)$] For real numbers $a$ and $b$, we denote
$\mathbb{N}_a=\{a,a+1,\ldots\}$ and ${_{b}\mathbb{N}}=\{b,b-1,\ldots\}$.

\item[$(iii)$] For $n \in \mathbb{N}$ and a real $a$, we denote
$_{\circleddash}\Delta^n f(t) = (-1)^n\Delta^n f(t)$,
$t \in \mathbb{N}_a$.

\item[$(iv)$] For $n \in \mathbb{N}$ and a real $b$, we denote
$\nabla_{\circleddash}^n f(t) = (-1)^n\nabla^n f(t)$,
$t \in {_{b}\mathbb{N}}$.
\end{description}
\end{notat}

Follows the definitions of delta/nabla left/right fractional sums.

\begin{defn}[See \cite{Th:dual:Riemann}]
\label{fractional sums}
Let $\sigma(t)=t+1$ and $\rho(t)=t-1$ be the forward
and backward jump operators, respectively. The delta
left fractional sum of order $\alpha>0$
(starting from $a$) is defined by
\begin{equation*}
\Delta_a^{-\alpha} f(t)=\frac{1}{\Gamma(\alpha)}
\sum_{s=a}^{t-\alpha}(t-\sigma(s))^{(\alpha-1)}f(s),
\quad t \in \mathbb{N}_{a+\alpha};
\end{equation*}
the delta right fractional sum of order
$\alpha>0$ (ending at $b$) by
\begin{equation*}
{_{b}\Delta^{-\alpha}} f(t)
=\frac{1}{\Gamma(\alpha)}
\sum_{s=t+\alpha}^{b}(s-\sigma(t))^{(\alpha-1)}f(s)\\
=\frac{1}{\Gamma(\alpha)}
\sum_{s=t+\alpha}^{b}(\rho(s)-t)^{(\alpha-1)}f(s),
\quad t \in {_{b-\alpha}\mathbb{N}};
\end{equation*}
the nabla left fractional sum of order $\alpha>0$
(starting from $a$) by
\begin{equation*}
\nabla_a^{-\alpha} f(t)=\frac{1}{\Gamma(\alpha)}
\sum_{s=a+1}^t(t-\rho(s))^{\overline{\alpha-1}}f(s),
\quad t \in \mathbb{N}_{a+1};
\end{equation*}
and the nabla right fractional sum of order
$\alpha>0$ (ending at $b$) by
\begin{equation*}
{_{b}\nabla^{-\alpha}} f(t)
=\frac{1}{\Gamma(\alpha)}
\sum_{s=t}^{b-1}(s-\rho(t))^{\overline{\alpha-1}}f(s)\\
=\frac{1}{\Gamma(\alpha)}
\sum_{s=t}^{b-1}(\sigma(s)-t)^{\overline{\alpha-1}}f(s),
\quad t \in {_{b-1}\mathbb{N}}.
\end{equation*}
\end{defn}

Regarding fractional sums, the next remarks are important.

\begin{rem}
\label{rem:a:f}
Operator $\Delta_a^{-\alpha}$ maps functions defined
on $\mathbb{N}_a$ to functions defined on $\mathbb{N}_{a+\alpha}$;
operator $_{b}\Delta^{-\alpha}$ maps functions defined on
$_{b}\mathbb{N}$ to functions defined on $_{b-\alpha}\mathbb{N}$;
$\nabla_a^{-\alpha}$ maps functions defined on $\mathbb{N}_a$
to functions defined on $\mathbb{N}_{a}$; while
$_{b}\nabla^{-\alpha}$ maps functions defined on $_{b}\mathbb{N}$
to functions on $_{b}\mathbb{N}$.
\end{rem}

\begin{rem}
Let $n \in \mathbb{N}$. Function $u(t)=\Delta_a^{-n}f(t)$
is solution to the initial value problem
$\Delta^n u(t)=f(t)$, $u(a+j-1)=0$,
$t\in \mathbb{N}_a$, $j=1,2,\ldots,n$;
function $u(t)={_{b}\Delta^{-n}}f(t)$
is solution to the initial value problem
$\nabla_\ominus^n u(t)=f(t)$, $u(b-j+1)=0$,
$t \in {_{b}\mathbb{N}}$, $j=1,2,\ldots,n$;
$\nabla_a^{-n}f(t)$ satisfies the $n$th order
discrete initial value problem
$\nabla^n y(t)=f(t)$, $\nabla^i y(a)=0$, $i=0,1,\ldots,n-1$;
while ${_{b}\nabla^{-n}}f(t)$ satisfies the $n$th order discrete
initial value problem ${_{\ominus}\Delta^n} y(t)=f(t)$,
$_{\ominus}\Delta^i y(b)=0$, $i=0,1,\ldots,n-1$.
\end{rem}

\begin{rem}
Consider the Cauchy functions
$f(t) = \frac{(t-\sigma(s))^{(n-1)}}{(n-1)!}$;
$g(t) = \frac{(\rho(s)-t)^{(n-1)}}{(n-1)!}$;
$h(t) = \frac{(t-\rho(s))^{\overline{n-1}}}{\Gamma(n)}$;
and $i(t)=\frac{(s-\rho(t))^{\overline{n-1}}}{\Gamma(n)}$.
Then, $f(t)$ vanishes at $s=t-(n-1),\ldots,t-1$;
$g(t)$ vanishes at $s=t+1$, $t+2$, $\ldots$, $t+(n-1)$;
$h(t)$ satisfies $\nabla^n y(t)=0$; and
$i(t)$ satisfies $_{\ominus}\Delta^n y(t)=0$.
\end{rem}

Now we recall the definitions of delta/nabla left/right fractional differences
in the sense of Riemann--Liouville. The definitions of Caputo fractional
differences, denoted by ${{^{C} \Delta}}$ and ${{^{C} \nabla}}$ instead of
$\Delta$ and $\nabla$, respectively, are not given here, and we refer
the reader to, e.g., \cite{ThCaputo,Th:dual:Caputo,MR3289943}.

\begin{defn}[See \cite{TDbyparts,Miller}]
\label{fractional differences}
The delta left fractional difference of order $\alpha>0$
(starting from $a$) is defined by
\begin{equation*}
\Delta_a^{\alpha} f(t)=\Delta^n \Delta_a^{-(n-\alpha)} f(t)
= \frac{\Delta^n}{\Gamma(n-\alpha)}
\sum_{s=a}^{t-(n-\alpha)}(t-\sigma(s))^{(n-\alpha-1)}f(s),
\quad t \in \mathbb{N}_{a+(n-\alpha)};
\end{equation*}
the delta right fractional difference of order $\alpha>0$ (ending at $b$) by
\begin{equation*}
{_{b}\Delta^{\alpha}} f(t)=  \nabla_{\circleddash}^n {_{b}\Delta^{-(n-\alpha)}}f(t)
=\frac{(-1)^n \nabla ^n}{\Gamma(n-\alpha)}
\sum_{s=t+(n-\alpha)}^{b}(s-\sigma(t))^{(n-\alpha-1)}f(s),
\quad t \in {_{b-(n-\alpha)}}\mathbb{N};
\end{equation*}
the nabla left fractional difference of order $\alpha>0$ (starting from $a$) by
\begin{equation*}
\nabla_a^{\alpha} f(t)=\nabla^n \nabla_a^{-(n-\alpha)}f(t)
= \frac{\nabla^n}{\Gamma(n-\alpha)}
\sum_{s=a+1}^t(t-\rho(s))^{\overline{n-\alpha-1}}f(s),
\quad t \in \mathbb{N}_{a+1},
\end{equation*}
and the nabla right fractional difference of order $\alpha>0$
(ending at $b$) is defined by
\begin{equation*}
{_{b}\nabla^{\alpha}} f(t)
= {_{\circleddash}\Delta^n} {_{b}\nabla^{-(n-\alpha)}}f(t)
=\frac{(-1)^n\Delta^n}{\Gamma(n-\alpha)}
\sum_{s=t}^{b-1}(s-\rho(t))^{\overline{n-\alpha-1}}f(s),
\quad t \in {_{b-1}\mathbb{N}}.
\end{equation*}
\end{defn}

Regarding the domains of the fractional differences, we observe
the following.

\begin{rem}
The delta left fractional difference $\Delta_a^\alpha$ maps functions defined
on $\mathbb{N}_a$ to functions defined on $\mathbb{N}_{a+(n-\alpha)}$;
the delta right fractional difference ${_{b}\Delta^\alpha}$ maps functions defined
on ${_{b}\mathbb{N}}$ to functions defined on ${_{b-(n-\alpha)}\mathbb{N}}$;
the nabla left fractional difference $\nabla_a^\alpha$ maps functions defined on
$\mathbb{N}_a$ to functions defined on $\mathbb{N}_{a+n}$; and the nabla right
fractional difference ${_{b}\nabla^\alpha}$ maps functions defined on
${_{b}\mathbb{N}}$ to functions defined on ${_{b-n}\mathbb{N}}$.
\end{rem}

\begin{lem}[See \cite{Ferd}]
\label{ATO}
If $\alpha >0$, then
$\Delta_a^{-\alpha} \Delta f(t)
=  \Delta \Delta_a^{-\alpha}f(t)
-\frac{(t-a)^{\overline{\alpha-1}}}{\Gamma(\alpha)} f(a)$.
\end{lem}

\begin{lem}[See \cite{TDbyparts}]
\label{TD}
If $\alpha >0$, then
${_{b} \Delta^{-\alpha}} \nabla_{\circleddash} f(t)
= \nabla_{\circleddash} {_{b} \Delta^{-\alpha}}f(t)
-\frac{(b-t)^{\overline{\alpha-1}}}{\Gamma(\alpha)} f(b)$.
\end{lem}

\begin{lem}[See \cite{Gronwall}]
\label{At}
If $\alpha >0$, then
$\nabla _{a+1}^{-\alpha} \nabla f(t)= \nabla \nabla_a^{-\alpha}f(t)
-\frac{(t-a+1)^{\overline{\alpha-1}}}{\Gamma(\alpha)}f(a)$.
\end{lem}

The result of Lemma~\ref{At} was obtained in \cite{Gronwall} by applying
the nabla left fractional sum starting from $a$ and not from $a+1$.
Lemma~\ref{AtT} provides a version of Lemma~\ref{At}
proved in \cite{Th:dual:Riemann}. Actually, the nabla fractional sums
defined in the articles \cite{Gronwall} and \cite{Th:dual:Riemann}
are related \cite{THFer}.

\begin{lem}[See \cite{Th:dual:Riemann}]
\label{AtT}
For any $\alpha >0$, the equality 
\begin{equation}
\label{AtT1}
\nabla _a^{-\alpha} \nabla f(t)
= \nabla \nabla_a^{-\alpha}f(t)
-\frac{(t-a)^{\overline{\alpha-1}}}{\Gamma(\alpha)}f(a)
\end{equation}
holds.
\end{lem}

\begin{rem}
\label{lforany}
Let $\alpha>0$ and $n=[\alpha]+1$. Then,
with the help of Lemma~\ref{AtT}, we have
\begin{equation*}
\nabla \nabla_a^\alpha f(t)=\nabla \nabla^n(\nabla_a^{-(n-\alpha)}f(t))
= \nabla^n (\nabla \nabla_a^{-(n-\alpha)}f(t))
\end{equation*}
or
\begin{equation*}
\nabla \nabla_a^\alpha f(t)=\nabla^n \left[\nabla_a^{-(n-\alpha)}\nabla
f(t)+\frac{(t-a)^{\overline{n-\alpha-1}}}{\Gamma(n-\alpha)}f(a)\right].
\end{equation*}
Then, using the identity
$\nabla^n \frac{(t-a)^{\overline{n-\alpha-1}}}{\Gamma(n-\alpha)}
=\frac{(t-a)^{\overline{-\alpha-1}}} {\Gamma(-\alpha)}$,
we infer that \eqref{AtT1} is valid for any real $\alpha$.
\end{rem}

With the help of Lemma~\ref{AtT}, Remark~\ref{lforany}, and the identity
$\nabla (t-a)^{\overline{\alpha-1}}=(\alpha-1)(t-a)^{\overline{\alpha-2}}$,
we arrive inductively to the following generalization.

\begin{thm}[See \cite{Th:dual:Caputo,THFer}]
\label{LNg}
For any real number $\alpha$ and any positive integer $p$, the
equality 
\begin{equation*}
\nabla_{a+p-1}^{-\alpha}~\nabla^p f(t)=\nabla^p
\nabla_{a+p-1}^{-\alpha}f(t)-\sum_{k=0}^{p-1}\frac{(t-(a+p-1))^{\overline{\alpha-p+k}}}
{\Gamma(\alpha+k-p+1)}\nabla^k f(a+p-1)
\end{equation*}
holds, where $f$ is defined on $\mathbb{N}_a$.
\end{thm}

\begin{lem}[See \cite{Th:dual:Riemann}]
\label{RN}
For any $\alpha >0$, the equality 
\begin{equation}
\label{RN1}
{_{b}\nabla^{-\alpha}} {_{\circleddash}\Delta} f(t)
= {_{\circleddash}\Delta}{_{b}\nabla^{-\alpha}} f(t)
-\frac{(b-t)^{\overline{\alpha-1}}}{\Gamma(\alpha)} f(b)
\end{equation}
holds.
\end{lem}

\begin{rem}
\label{forany}
Let $\alpha>0$ and $n=[\alpha]+1$. Then,
with the help of Lemma~\ref{RN}, we have
\begin{equation*}
{_{a}\Delta}{_{b}\nabla^\alpha} f(t)={_{a}\Delta}
{_{\circleddash}\Delta^n}({_{b}\nabla^{-(n-\alpha)}}f(t))
={_{\circleddash}\Delta^n}({_{\circleddash}\Delta}{_{b}\nabla^{-(n-\alpha)}}f(t))
\end{equation*}
or
\begin{equation*}
{_{\circleddash}\Delta} {_{b}\nabla^\alpha} f(t)
={_{\circleddash}\Delta^n} \left[{_{b}\nabla^{-(n-\alpha)}}{_{\circleddash}\Delta}
f(t)+\frac{(b-t)^{\overline{n-\alpha-1}}} {\Gamma(n-\alpha)}f(b)\right].
\end{equation*}
Then, using the identity
${_{\circleddash}\Delta^n}\frac{(b-t)^{\overline{n-\alpha-1}}}{\Gamma(n-\alpha)}
=\frac{(b-t)^{\overline{-\alpha-1}}} {\Gamma(-\alpha)}$,
we infer that \eqref{RN1} is valid for any real $\alpha$.
\end{rem}

With the help of Lemma~\ref{RN}, Remark~\ref{forany}, and the identity
$\Delta (b-t)^{\overline{\alpha-1}}=-(\alpha-1)(b-t)^{\overline{\alpha-2}}$,
we arrive by induction to the following generalization.

\begin{thm}[See \cite{Th:dual:Caputo,THFer}]
\label{RNg}
For any real number $\alpha$ and any positive integer $p$,
the equality
\begin{equation*}
{~_{b-p+1}\nabla^{-\alpha}} {_{\circleddash}\Delta^p} f(t)
={_{\circleddash}\Delta^p} {~_{b-p+1}\nabla^{-\alpha}}f(t)
-\sum_{k=0}^{p-1}\frac{(b-p+1-t)^{\overline{\alpha-p+k}}}
{\Gamma(\alpha+k-p+1)}{~_{\ominus}\Delta^k} f(b-p+1)
\end{equation*}
holds, where $f$ is defined on $_{b}\mathbb{N}$.
\end{thm}


\section{Symmetric duality for left and right
Riemann--Liouville and Caputo fractional differences}
\label{sec:mr}

In this section, we use the recent notion of duality for the continuous
fractional calculus, as introduced by Caputo and Torres in \cite{Caputo:Torres},
to prove symmetric duality identities for delta and nabla fractional sums
and differences. The next result (as well as Theorem~\ref{thm:dual:delta:sums}),
shows that the left fractional sum of a given function $f$
is the right fractional sum of the dual of $f$.

\begin{thm}[Symmetric duality of nabla fractional sums]
\label{T}
Let $f:\mathbb{N}_a \cap {_{b}\mathbb{N}} \rightarrow \mathbb{R}$ be a given
function and $f^*:\mathbb{N}_{-b} \cap {_{-a}\mathbb{N}} \rightarrow \mathbb{R}$
be its symmetric dual, that is, $f^*(t)=f(-t)$. Then,
\begin{equation}
\label{eq:y1}
(\nabla_a^{-\alpha} f)(t)= (_{-a}\nabla ^{-\alpha} f^*)(-t),
\end{equation}
where on the right-hand side of \eqref{eq:y1} we have the nabla right fractional
sum of $f^*$ ending at $-a$ and evaluated at $-t$.
\end{thm}

\begin{proof}
Using Definition~\ref{fractional sums},
and the change of variable $s=-u$, we have
\begin{equation*}
\begin{split}
(\nabla_a^{-\alpha} f)(t)
&= \frac{1}{\Gamma(\alpha)}
\sum_{u=a+1}^t (t-\rho(u) )^{\overline{\alpha-1}} f(u) \\
&= - \frac{1}{\Gamma(\alpha)}
\sum_{s=-a-1}^{-t} (t+s+1)^{\overline{\alpha-1}}f(-s)\\
&= \frac{1}{\Gamma(\alpha)}
\sum_{s=-t}^{-a-1} (s-\rho(-t))^{\overline{\alpha-1} }f^*(s)\\
&= ({_{-a}\nabla^{-\alpha}} f^*)(-t).
\end{split}
\end{equation*}
This concludes the proof.
\end{proof}

While in the fractional case symmetric duality relates left and right operators,
the next two results (Lemma~\ref{x} and Theorem~\ref{nx}) show that in the
integer-order case the symmetric duality relates delta and nabla operators.
This is a consequence of the general duality on time scales \cite{Caputo}.

\begin{lem}[Symmetric duality of forward and backward difference operators]
\label{x}
Let $f:\mathbb{N}_a \cap {_{b}\mathbb{N}} \rightarrow \mathbb{R}$ and
$f^*:\mathbb{N}_{-b} \cap {_{-a}\mathbb{N}} \rightarrow \mathbb{R}$ be its
symmetric dual function. Then,
\begin{equation}
\label{eq:ND}
-(\nabla f)^*(t)=\Delta f^*(t)
\end{equation}
and
\begin{equation}
\label{eq:DN}
-(\Delta f)^*(t)=\nabla f^*(t).
\end{equation}
\end{lem}

\begin{proof}
Let $g(t)=\nabla f(t)=f(t)-f(t-1)$. Then,
$$
-g^*(t)=-g(-t)=f(-t-1)-f(-t)=f(-(t+1))-f(-t)=\Delta f^*(t)
$$
and \eqref{eq:ND} is proved. The proof of \eqref{eq:DN} is similar
by defining $h(t)=\Delta f(t)=f(t+1)-f(t)$.
\end{proof}

Relations \eqref{eq:ND} and \eqref{eq:DN} are easily
generalized to the higher-order case.

\begin{thm}[Symmetric duality of integer-order difference operators]
\label{nx}
Let $n \in \mathbb{N}$,
$f:\mathbb{N}_a \cap {_{b}\mathbb{N}} \rightarrow \mathbb{R}$ be a given
function and $f^*:\mathbb{N}_{-b} \cap {_{-a}\mathbb{N}} \rightarrow \mathbb{R}$
be its symmetric dual. Then,
$$
(\nabla^n f)^*(t) = (-1)^n \Delta^n f^*(t)= {_{\ominus}\Delta^n} f^*(t)
$$
and
$$
(\Delta^n f)^*(t) = (-1)^n \nabla^n f^*(t)= \nabla_{\ominus}^n f^*(t).
$$
\end{thm}

\begin{proof}
The case $n=1$ is true from Lemma~\ref{x}. We obtain the intended result
by induction.
\end{proof}

Our previous results allow us to relate left and right nabla fractional differences.

\begin{thm}[Symmetric duality of Riemann--Liouville nabla fractional difference operators]
\label{ndual}
Assume $f:\mathbb{N}_a \cap {_{b}\mathbb{N}} \rightarrow \mathbb{R}$
and let $f^*:\mathbb{N}_{-b} \cap {_{-a}\mathbb{N}} \rightarrow \mathbb{R}$
be its symmetric dual. Then, for each $n-1<\alpha \leq n$, $n \in \mathbb{N}_1$,
we have that the left fractional difference of $f$ starting at $a$ and evaluated
at $t$ is the right fractional difference of $f^*$ ending at $-a$ and evaluated at $-t$:
\begin{equation*}
(\nabla_a^\alpha f)(t)= ({_{-a}\nabla^\alpha} f^*)(-t).
\end{equation*}
\end{thm}

\begin{proof}
By Definition~\ref{fractional differences}, and the help
of Theorems~\ref{T} and \ref{nx}, it follows that
\begin{equation*}
\begin{split}
(\nabla_a^\alpha f)(t)
&= \nabla^n (\nabla_a^{-(n-\alpha)}f)(t)\\
&=\nabla^n ({-a} \nabla^{-(n-\alpha)}f^*)(-t) \\
&= \nabla^n ({_{-a}\nabla^{-(n-\alpha)}}f^*)^*(t)\\
&= {_{\ominus}\Delta^n} ({_{-a}\nabla^{-(n-\alpha)}}f^*)^*(t)\\
&= {_{\ominus}\Delta^n}(~_{-a}\nabla^{-(n-\alpha)}f^*)(-t)\\
&= ({_{-a}\nabla^\alpha} f^*)(-t).
\end{split}
\end{equation*}
This completes the proof.
\end{proof}

An analogous result to Theorem~\ref{ndual} also holds for fractional differences
in the sense of Caputo.

\begin{thm}[Symmetric duality of Caputo nabla fractional difference operators]
\label{CC}
Given a function $f:\mathbb{N}_a \cap {_{b}\mathbb{N}} \rightarrow \mathbb{R}$,
let $f^*:\mathbb{N}_{-b} \cap {_{-a}\mathbb{N}} \rightarrow \mathbb{R}$ be
its symmetric dual. Then, for each $n-1<\alpha \leq n$,
$n \in \mathbb{N}_1$, we have
\begin{equation*}
({^{C}\nabla_{a(\alpha)}^\alpha} f)(t)
= ({_{-a(\alpha)}^{C}\nabla^\alpha} f^*)(-t),
\end{equation*}
where $a(\alpha)=a+n-1$.
\end{thm}

\begin{proof}
Using the definition of Caputo fractional differences,
Theorems~\ref{T} and \ref{nx}, we have
\begin{equation*}
\begin{split}
({^{C}\nabla_{a(\alpha)}^\alpha} f)(t)
&= (\nabla_{a(\alpha)} ^{-(n-\alpha)} \nabla^n f)(t)\\
&=({_{a(\alpha)}\nabla^{-(n-\alpha)}}(\nabla^nf(t))^*)(-t)\\
&=({_{a(\alpha)}\nabla^{-(n-\alpha)}}({_{\ominus}\Delta^n}f^*(t))^*)(-t)\\
&=({^{C}_{-a(\alpha)}\nabla^\alpha} f^*)(-t).
\end{split}
\end{equation*}
The proof is complete.
\end{proof}

We now show that symmetric duality results for delta fractional sums
and delta fractional differences can be achieved from our previous
duality results on nabla fractional operators by using the
approach in \cite{TDAF}. For delta fractional sums and differences
we make use of the next two lemmas summarized
and cited accurately in \cite{TDAF}.

\begin{lem}[See \cite{TDAF}]
\label{left dual}
Let $0\leq n-1< \alpha \leq n$ and let $y(t)$ be defined on $\mathbb{N}_a$.
Then, the following statements are valid:
\begin{description}
\item[(i)] $(\Delta_a^\alpha) y(t-\alpha)
= \nabla_{a-1}^\alpha y(t)$ for $t \in \mathbb{N}_{n+a}$;

\item[(ii)] $(\Delta_a^{-\alpha}) y(t+\alpha)=\nabla_{a-1}^{-\alpha} y(t)$
for $t \in \mathbb{N}_a$.
\end{description}
\end{lem}

\begin{lem}[See \cite{TDAF}]
\label{right dual}
Let $y(t)$ be defined on ${_{b+1}\mathbb{N}}$.
The following statements are valid:
\begin{description}
\item[(i)] $({_{b}\Delta^\alpha})y(t+\alpha)
= {_{b}\nabla^\alpha} y(t)$ for $t \in {_{b-n}\mathbb{N}}$;

\item[(ii)] $({_{b}\Delta^{-\alpha}}) y(t-\alpha)= {_{b+1}\nabla^{-\alpha}} y(t)$
for $t \in {_{b}\mathbb{N}}$.
\end{description}
\end{lem}

Our next result is the delta analogous of Theorem~\ref{T}.

\begin{thm}[Symmetric duality of delta fractional sums]
\label{thm:dual:delta:sums}
Let $f:\mathbb{N}_a \cap {_{b}\mathbb{N}} \rightarrow \mathbb{R}$,
$a<b$, and $f^*:\mathbb{N}_{-b} \cap {_{-a}\mathbb{N}} \rightarrow \mathbb{R}$
be its symmetric dual function. Then, for each $n-1<\alpha \leq n$,
$n \in \mathbb{N}_1$, and $t \in \mathbb{N}_a$, we have
\begin{equation*}
\Delta_a^{-\alpha} f(\sigma^\alpha(t))
= ({_{-a}\Delta^{-\alpha}}f^*)(-\sigma^\alpha(t)),
\end{equation*}
where the fractional forward jump operator $\sigma^\alpha$
is defined by $\sigma^\alpha(t) = t+\alpha$.
\end{thm}

\begin{proof}
The following relations hold:
\begin{equation*}
\begin{split}
\Delta_a^{-\alpha} f(t+\alpha)
&= \nabla^{-\alpha}_{a-1} f(t)\\
&= ({_{-a+1}\nabla^{-\alpha}} f^*)(-t)\\
&= ({_{-a}\Delta^{-\alpha}}f^*)(-t-\alpha).
\end{split}
\end{equation*}
The above equalities follow by Lemma~\ref{left dual} (ii),
Theorem~\ref{T}, and Lemma~\ref{right dual} (ii) (with $b=-a$).
\end{proof}

We now prove symmetric duality for delta fractional differences.

\begin{thm}[Symmetric duality of Riemann--Liouville delta fractional difference operators]
\label{thm:bsdds}
Let $f:\mathbb{N}_a \cap {_{b}\mathbb{N}} \rightarrow \mathbb{R}$, $a<b$,
and $f^*:\mathbb{N}_{-b} \cap {_{-a}\mathbb{N}} \rightarrow \mathbb{R}$
be its symmetric dual function. Then, for each $n-1<\alpha \leq n$,
$n \in \mathbb{N}_1$, and $t \in \mathbb{N}_{a+n}$, we have
\begin{equation*}
\Delta_a^{\alpha} f(\rho^\alpha(t))
= ({_{-a}\Delta^{\alpha}}f^*)(-\rho^\alpha(t)),
\end{equation*}
where the fractional backward jump operator $\rho^\alpha$
is defined by $\rho^\alpha(t) = t-\alpha$.
\end{thm}

\begin{proof}
By Lemma~\ref{left dual} (i),
Theorem~\ref{ndual}, and Lemma~\ref{right dual} (i)
(with $b=-a$), it follows that
\begin{equation*}
\begin{split}
\Delta_a^{\alpha} f(t-\alpha)
&= \nabla^{\alpha}_{a-1} f(t)\\
&= {_{-a+1}\nabla^{\alpha}} f^*(-t)\\
&= ({_{-a}\Delta^{\alpha}}f^*)(-t+\alpha).
\end{split}
\end{equation*}
The result is proved.
\end{proof}

The next proposition states a relation between
left delta Caputo fractional differences
and left nabla Caputo fractional differences.

\begin{prop}[See \cite{Th:dual:Caputo}]
\label{lCdual}
For $f:\mathbb{N}_a\rightarrow \mathbb{R}$, $\alpha >0$,
$n=[\alpha]+1$, $a(\alpha)=a+n-1$, we have
\begin{equation*}
({^{C}\Delta_a^\alpha} f)(t-\alpha)
= ({^{C}\nabla_{a(\alpha)}^\alpha} f)(t),
\quad t \in \mathbb{N}_{a+n}.
\end{equation*}
\end{prop}

Analogously, the following proposition relates right delta Caputo fractional
differences and right nabla Caputo fractional differences.

\begin{prop}[See \cite{Th:dual:Caputo}]
\label{rCdual}
For $f:{_{b}\mathbb{N}}\rightarrow \mathbb{R}$,
$\alpha >0$, $n=[\alpha]+1,~b(\alpha)=b-n+1$, we have
\begin{equation*}
({^{C}_{b}\Delta^\alpha} f)(t+\alpha)
= ({^{C}_{b(\alpha)}\nabla^\alpha} f)(t),
\quad t \in {_{b-n}\mathbb{N}}.
\end{equation*}
\end{prop}

Using Propositions~\ref{lCdual} and \ref{rCdual}, as well as our Theorem~\ref{CC}
on the symmetric duality of the Caputo nabla fractional difference operators,
we prove a symmetric duality result for the delta fractional differences.

\begin{thm}[Symmetric duality of Caputo delta fractional difference operators]
Let $f:\mathbb{N}_a \cap {_{b}\mathbb{N}} \rightarrow \mathbb{R}$, $a<b$,
and $f^*:\mathbb{N}_{-b} \cap {_{-a}\mathbb{N}} \rightarrow \mathbb{R}$
be its symmetric dual. Then, for each $n-1<\alpha \leq n$, $n \in \mathbb{N}_1$,
and $t \in \mathbb{N}_{a+n}$, we have
\begin{equation*}
{^{C}\Delta_a^{\alpha}} f(\rho^\alpha(t))
= ({^{C}_{-a}\Delta^{\alpha}}f^*)(-\rho^\alpha(t)),
\end{equation*}
where $\rho^\alpha(t) = t-\alpha$.
\end{thm}

\begin{proof}
From Proposition~\ref{lCdual}, Theorem~\ref{CC},
and Proposition~\ref{rCdual}, we have
\begin{equation*}
\begin{split}
{{^{C}\Delta_a^{\alpha}}} f(t-\alpha)
&= {~^{C}\nabla^{\alpha}_{a(\alpha)}} f(t)\\
&= {{^{C}_{-a(\alpha)}\nabla^\alpha} f^*}(-t)\\
&= ({{^{C}_{-a}\Delta^{\alpha}}}f^*)(-t+\alpha).
\end{split}
\end{equation*}
The result follows by using the definition
of the fractional backward jump operator $\rho^\alpha(t)$.
\end{proof}


\section{Summation by parts for fractional differences}
\label{sec:appl1}

The next version of fractional ``integration by parts'' with boundary
conditions was proved in \cite[Theorem~43]{Th:dual:Caputo}. It has then been
used in \cite{ThEuler} to obtain a discrete fractional variational principle.

\begin{thm}[See \cite{Th:dual:Caputo}]
\label{Caputo by parts}
Let $0<\alpha<1$ and $f$ and $g$ be two functions defined on
$\mathbb{N}_a \cap {_{b}\mathbb{N}}$, where $a \equiv b\mod 1$. Then,
\begin{equation*}
\sum_{s=a+1}^{b-1} g(s) {^{C}\nabla_a^\alpha} f(s)
=f(s) {_{b}\nabla^{-(1-\alpha)}}g(s)\mid_a^{b-1}
+ \sum_{s=a+1}^{b-1} f(s-1) ({_{b}\nabla^\alpha} g)(s-1),
\end{equation*}
where $(_{b}\nabla^{-(1-\alpha)}g)(b-1)= g(b-1)$.
\end{thm}

If we interchange the role of Caputo and Riemann--Liouville operators, we obtain
the following version of summation by parts for fractional differences.
It was proved in \cite{ThEuler} and used there to obtain a discrete fractional
Euler--Lagrange equation.

\begin{thm}[See \cite{ThEuler}]
\label{Riemann2 by parts}
Let $0<\alpha<1$ and $f$ and $g$ be functions defined on
$\mathbb{N}_a \cap {_{b}\mathbb{N}}$, where $a\equiv b\mod 1$. Then,
\begin{equation*}
\begin{split}
\sum_{s=a+1}^{b-1} f(s-1) \nabla_a^\alpha g(s)
&= f(s) \nabla_a^{-(1-\alpha)}g(s)\mid_a^{b-1}
+ \sum_{s=a}^{b-2} g(s+1) \, {^{C}_{b}\nabla^\alpha} f(s)\\
&= f(s) \nabla_a^{-(1-\alpha)}g(s)\mid_a^{b-1}
+ \sum_{s=a+1}^{b-1} g(s) \, ({^{C}_{b}\nabla^\alpha} f)(s-1),
\end{split}
\end{equation*}
where $\nabla_a^{-(1-\alpha)}g(a)= 0$.
\end{thm}

The above formulas of fractional summation by parts were only obtained for
left fractional differences. In this section, we use the symmetric duality
results of Section~\ref{sec:mr} to obtain new summation by parts formulas
for the right fractional differences.

\begin{thm}
\label{Ap1}
Let $0<\alpha<1$ and $f$ and $g$ be functions defined on
$\mathbb{N}_a \cap {_{b}\mathbb{N}}$, where $a\equiv b\mod 1$. Then,
\begin{equation*}
\sum_{s=a+1}^{b-1} g(s) \, {{^{C}_{b}\nabla^\alpha}} f(s)
=f(s) \nabla_a^{-(1-\alpha)}g(s)\mid_{b}^{a+1}
+ \sum_{s=a+1}^{b-1} f(s+1) (\nabla_a^\alpha g)(s+1),
\end{equation*}
where $\left(\nabla_a^{-(1-\alpha)}g\right)(a+1)= g(a+1)$.
\end{thm}

\begin{proof}
First, we note that if we apply Theorem~\ref{CC} to $f^*$ starting at $-b$,
$a(\alpha)=a$, $n=1$, and using the fact that $f^{**}=f$, we conclude that
\begin{equation}
\label{see that1}
({^{C}\nabla_{-b}^\alpha} f^*)(t)
= ({^{C}_{b}\nabla^\alpha} f^{**})(-t)
=({^{C}_{b}\nabla^\alpha} f)(-t).
\end{equation}
Then, by the change of variable $s=-t$, it follows from \eqref{see that1} that
\begin{equation*}
\begin{split}
\sum_{s=a+1}^{b-1} g(s) \, {{^{C}_{b}\nabla^\alpha}}f(s)
&= -\sum_{t=-a-1}^{-b+1}g(-t) \, {{^{C}_{b}\nabla^\alpha}}f(-t)\\
&= \sum_{t=-b+1}^{-a-1}g^*(t)({{^{C}_{b}\nabla^\alpha}} f)(-t)\\
&=\sum_{t=-b+1}^{-a-1}g^*(t) ({{^{C}\nabla_{-b}^\alpha}} f^*)(t).
\end{split}
\end{equation*}
Applying Theorem~\ref{Caputo by parts} to the pair $(f^*,g^*)$
with $a \rightarrow -b$ and $b \rightarrow -a$, we reach at
\begin{equation*}
\begin{split}
\sum_{s=a+1}^{b-1} g(s)\, {{^{C}_{b}\nabla^\alpha}} f(s)
&= f^*(s)({_{-a}\nabla^{-(1-\alpha)}}g^*)(s)|_{-b}^{-a-1}
+\sum_{s=-b+1}^{-a-1}f^*(s-1)({_{-a}\nabla^\alpha} g^*)(s-1)\\
&= f^*(s)({_{-a}\nabla^{-(1-\alpha)}}g^*)(s)|_{-b}^{-a-1}
+\sum_{s=-b}^{-a-2}f^*(s)({_{-a}\nabla^\alpha} g^*)(s) \\
&= f(s)({_{-a}\nabla^{-(1-\alpha)}}g^*)(-s)|_{b}^{a+1}
+\sum_{s=a+2}^{b}f(s)({_{-a}\nabla^\alpha} g^*)(-s).
\end{split}
\end{equation*}
Then, by Theorem~\ref{T} and Theorem~\ref{ndual}, we have
\begin{equation*}
\begin{split}
\sum_{s=a+1}^{b-1} g(s) \, {{^{C}_{b}\nabla^\alpha}}f(s)
&= f(s)(\nabla_{a}^{-(1-\alpha)}g)(s)|_{b}^{a+1}
+\sum_{s=a+2}^{b}f(s)(\nabla_a^\alpha g)(s)\\
&= f(s)(\nabla_{a}^{-(1-\alpha)}g)(s)|_{b}^{a+1}
+\sum_{s=a+1}^{b}f(s+1)(\nabla_a^\alpha g)(s+1)
\end{split}
\end{equation*}
and the result is proved.
\end{proof}

Theorem~\ref{Ap1} relates a right Caputo nabla difference with
a left Riemann--Liouville nabla difference. In contrast,
the next theorem relates a right Riemann--Liouville nabla difference
with a left Caputo nabla difference.

\begin{thm}
\label{Ap2}
Let $0<\alpha<1$ and $f$ and $g$ be functions defined on
$\mathbb{N}_a \cap {_{b}\mathbb{N}}$, where $a\equiv b\mod 1$. Then,
\begin{equation*}
\sum_{s=a+1}^{b-1} f(s-1) {_{b}\nabla^\alpha} g(s)
=f(s) {_{b} \nabla^{-(1-\alpha)}}g(s)\mid_{b}^{a+1}
+ \sum_{s=a+1}^{b-1} g(s) ({^{C}\nabla_a^\alpha} f)(s+1),
\end{equation*}
where $({_{b}\nabla^{-(1-\alpha)}}g)(b)= 0$.
\end{thm}

\begin{proof}
First, we apply Theorem~\ref{ndual} for $f^*$ starting at $-b$, $n=1$.
Using the relation $f^{**}=f$, we see that
\begin{equation}
\label{see that2}
(\nabla_{-b}^\alpha f^*)(t)
= ({_{b}\nabla^\alpha} f^{**})(-t)
=({_{b}\nabla^\alpha} f)(-t).
\end{equation}
Then, by the change of variable $s=-t$ and the help of \eqref{see that2}, we have
\begin{equation*}
\begin{split}
\sum_{s=a+1}^{b-1} f(s-1){_{b}\nabla^\alpha} g(s)
&= -\sum_{t=-a-1}^{-b+1}f(-t-1){_{b}\nabla^\alpha} g(-t)\\
&= \sum_{t=-b+1}^{-a-1} f^*(t-1)({_{b}\nabla^\alpha} g^*)(-t)\\
&=\sum_{t=-b+1}^{-a-1}f^*(t-1)(\nabla_{-b}^\alpha g^*)(t).
\end{split}
\end{equation*}
Applying Theorem~\ref{Riemann2 by parts} to the pair $(f^*,g^*)$
with $a \rightarrow -b$ and $b \rightarrow -a$, we reach at
\begin{equation*}
\begin{split}
\sum_{s=a+1}^{b-1} f(s-1){_{b}\nabla^\alpha} g(s)
&= f^*(s)(\nabla_{-b}^{-(1-\alpha)}g^*)(s)|_{-b}^{-a-1}
+\sum_{s=-b}^{-a-2}g^*(s-1)({^{C}_{-a}\nabla^\alpha} f^*)(s)\\
&= f(s)(\nabla_{-b}^{-(1-\alpha)}g^*)(-s)|_{b}^{a+1}
+\sum_{s=a+2}^{b}g^*(-s-1)({^{C}_{-a}\nabla^\alpha} f^*)(-s).
\end{split}
\end{equation*}
Then, by Theorem~\ref{T} and Theorem~\ref{CC}, we have
\begin{equation*}
\begin{split}
\sum_{s=a+1}^{b-1} f(s-1){_{b}\nabla^\alpha} g(s)
&= f(s)({_{b}\nabla^{-(1-\alpha)}}g)(s)|_{b}^{a+1}
+\sum_{s=a+2}^{b}g(s-1)({^{C}\nabla_a^\alpha} f)(s)\\
&= f(s)({_{b}\nabla^{-(1-\alpha)}}g)(s)|_{b}^{a+1}
+\sum_{s=a+1}^{b-1}g(s)({^{C}\nabla_a^\alpha} f)(s+1).
\end{split}
\end{equation*}
This concludes the proof.
\end{proof}


\section{Application to the discrete fractional variational calculus}
\label{sec:appl2}

The study of discrete fractional variational problems is mainly concentrated
on obtaining Euler--Lagrange equations for a minimizer containing left fractional
differences \cite{MR2728463,Nuno,MR2809039}. Roughly speaking, after applying
a fractional summation by parts formula, one is able to express Euler--Lagrange
equations by means of right fractional differences. In this section, we reverse
the order and we start by minimizing a discrete functional containing right
fractional differences. We will then use the summation by parts formulas obtained
in Section~\ref{sec:appl1} to prove left versions of the fractional difference
Euler--Lagrange equations obtained in \cite{ThEuler}
(cf. Theorems~3.2 and 3.3 in \cite{ThEuler}).

\begin{thm}
\label{mm}
Let $0< \alpha <1$ be noninteger, $a,b \in \mathbb{R}$, and $f$ be defined on
$\mathbb{N}_a \cap {_{b}\mathbb{N}}$, where $a\equiv b \mod 1$. Assume that
the discrete functional
$$
J(y)=\sum_{t=a+1}^{b-1} L\left(t,f(t),{_{b}\nabla^\alpha} f(t)\right)
$$
has a local extremizer in $S=\{y:\mathbb{N}_a \cap {_{b}\mathbb{N}}
\rightarrow \mathbb{R}~\text{is bounded}\}$ at some $f \in S$, where
$L:(\mathbb{N}_a \cap {_{b}\mathbb{N}})\times \mathbb{R}
\times \mathbb{R}\rightarrow \mathbb{R}$. Further, assume that either
${_{b}\nabla^{-(1-\alpha)}}f (a+1)=B$ or $L_2^\sigma (a+1)=0$. Then,
\begin{equation*}
\left[L_1(s) + ({^{C}\nabla_a^\alpha} L_2^\sigma)(s+1)\right]=0
\quad \text{for all} \quad s \in \mathbb{N}_{a+1} \cap {_{b-1}\mathbb{N}},
\end{equation*}
where $L_1(s)= \frac{\partial L}{\partial f}(s)$,
$L_2(s)=\frac{\partial L}{\partial {_{b}\nabla^\alpha} f}(s)$
and $L_2^\sigma(t)=L_2(\sigma(t))$.
\end{thm}

\begin{proof}
Similar to the proof of Theorem~3.2 in \cite{ThEuler}
by making use of our Theorem~\ref{Ap2}.
\end{proof}

Finally, we obtain the Euler--Lagrange equation for a Lagrangian depending
on the Caputo right fractional difference and by making use of the summation
by parts formula in Theorem~\ref{Ap1}.

\begin{thm}
\label{mmm}
Let $0< \alpha <1$ be noninteger, $a,b \in \mathbb{R}$, and $f$ be defined on
$\mathbb{N}_a \cap {_{b}\mathbb{N}}$, where $a\equiv b \mod 1$. Assume that
the discrete functional
$$
J(f)=\sum_{t=a+1}^{b-1} L\left(t,f^\sigma(t), {^{C}_b\nabla^\alpha} f(t)\right)
$$
has a local extremizer in $S=\{y:\mathbb{N}_a \cap {_{b}\mathbb{N}}
\rightarrow \mathbb{R}~\text{is bounded}\}$ at some $f \in S$, where
$L:(\mathbb{N}_a \cap {_{b}\mathbb{N}})\times \mathbb{R}\times \mathbb{R}
\rightarrow \mathbb{R}$. Further, assume that either $f(a+1)=C$ and $f(b)=D$
or the natural boundary conditions $\nabla_a^{-(1-\alpha)}L_2 (a+1)
=\nabla_a^{-(1-\alpha)}L_2 (b)=0$ hold. Then,
\begin{equation*}
\left[L_1(s) + (\nabla_a^\alpha L_2)(\sigma(s))\right] =0
\quad \text{for all} \quad s \in \mathbb{N}_{a+2} \cap {_{b-1}\mathbb{N}}.
\end{equation*}
\end{thm}

\begin{proof}
Similar to the proof of Theorem~3.3 in \cite{ThEuler} by making use
of our Theorem~\ref{Ap1}.
\end{proof}


\section*{Acknowledgements}

Torres was partially supported by the Portuguese Foundation for Science
and Technology (FCT), through the Center for Research and Development
in Mathematics and Applications (CIDMA), within project UID/MAT/04106/2013.



\end{document}